\documentclass{nyjm}

\usepackage{amsmath}
\usepackage{amssymb}
\usepackage{epsfig}  		
\usepackage{xypic}
\usepackage{bbm}
\usepackage{hyperref}
\usepackage{eucal}
\usepackage{mathrsfs}

\newtheorem{mthm}{Theorem}

\newtheorem*{mcor}{Corollary}

\newtheorem{thm}{Theorem}[section]
\newtheorem{prop}[thm]{Proposition}
\newtheorem{lem}[thm]{Lemma}

\theoremstyle{definition}

\theoremstyle{remark}

\newtheorem{remark}[thm]{Remark}

\numberwithin{equation}{section}

\newcommand{\RR}{\mathbbm{R}}

\newcommand{\DD}{\mathfrak{D}} 
\renewcommand{\AA}{\mathfrak{A}}

\newcommand{\group}{\mathsf}

\newcommand{\Iso}{\group{Iso}}
\newcommand{\Aff}{\group{Aff}}

\newcommand{\OO}{\group{O}}

\newcommand{\Hol}{\group{Hol}}
\newcommand{\Zen}{\mathrm{C}}

\DeclareMathOperator{\im}{\mathrm{im}}
\DeclareMathOperator{\rk}{\mathrm{rk}}

\begin{document}

\title[Flat Homogeneous Spaces of Signature $(m,2)$]{A Supplement to the Classification of Flat Homogeneous Spaces of Signature $(m,2)$}

\author[Globke]{Wolfgang Globke}

\address{Wolfgang Globke,
School of Mathematical Sciences\\
The University of Adelaide\\
SA 5005\\
Australia}
\email{wolfgang.globke@adelaide.edu.au}

\thanks{The author was supported by the Australian Research Council grant DP120104582.}

\begin{abstract}
Duncan and Ihrig (1993) gave a classification of the
flat homo\-geneous spaces of metric
signature $(m,2)$, provided that a certain condition on
the development image of these spaces holds.
In this note we show that this condition can be
dropped, so that Duncan and Ihrig's classification is in fact
the full classification for signature $(m,2)$.
\end{abstract}

\maketitle

\section{Introduction}

In a series of three articles,
Duncan and Ihrig \cite{DI1, DI2, DI3}
developed a theory for
flat pseudo-Riemannian homogeneous spaces.
For geodesically complete manifolds, this theory was
previously developed by Wolf \cite{wolf1}, and his methods
were extended by Duncan and Ihrig to the geodesically
incomplete case.

In their first article \cite{DI1}, they classified the
Lorentzian flat homogeneous spaces.
In the second article \cite{DI2}, they studied isometry groups
of pseudo-Riemannian flat homogeneous spaces.
Underlying their investigations were the following facts
(see Goldman and Hirsch \cite{GH}):
Let $M$ be a flat homogeneous pseudo-Riemannian manifold
of metric signature $(n-s,s)$.
If $M$ is geodesically complete, then $M=\RR^n_s/\Gamma$,
where $s$ is the index of a pseudo-scalar product $\langle\cdot,\cdot\rangle$
on $\RR^n$, and $\Gamma\subset\Iso(\RR^n_s)$ is the
fundamental group of $M$.
If $M$ is not geodesically complete, then its universal cover
$\tilde{M}$ can be mapped to an affine homogeneous domain
$\DD\subset\RR^n_s$ via the development map. We call $\DD$
the \emph{development image} of $\tilde{M}$.
The development map induces a homomorphism from the fundamental
group $\pi_1(M)$ onto a group $\Gamma\subset\Iso(\DD)$,
which we will call the \emph{affine holonomy group} of $M$.
Then the manifold $M$ can be realised as $M=\DD/\Gamma$.

In their third article \cite{DI3}, Duncan and Ihrig classified
those flat homogeneous spaces of signature $(n-2,2)$
whose associated domain $\DD$ is \emph{translationally
isotropic}.
This means that the set $T$ of all translations in $\RR^n_2$
leaving $\DD$ invariant contains its own orthogonal space:
$T^\perp\subset T$.
We will prove in this article that every development image
of a flat homogeneous space of signature $(n-2,2)$
is translationally isotropic.
As a consequence, the classification by Duncan and Ihrig
\cite[Chapter 3]{DI3} is already the full classification
of flat homogeneous spaces of  signature $(n-2,2)$.

\begin{remark}\label{rem_correction}
In their classification, Duncan and Ihrig
\cite[Theorem 3.1, case I]{DI3} misquote the geodesically complete case
from Wolf's article \cite{wolf1}.
They state that in this case the affine holonomy group $\Gamma$
is a discrete group of pure translations. However, it is only true
that $\Gamma$ is a free abelian group in this case, but
as Wolf's example in Section 6 of \cite{wolf1} shows,
the elements of $\Gamma$ can have non-trivial linear parts.
The classification for this case is given by Wolf \cite[Theorem 3.6]{wolf2}.
\end{remark}

Our main result is the following theorem:

\begin{mthm}\label{thm_abhol_transiso}
Let $M=\DD/\Gamma$ be a flat pseudo-Riemannian homogeneous manifold, where
$\DD\subseteq\RR^n_s$ is an open orbit of the centraliser of $\Gamma$ in
$\Iso(\RR^n_s)$.
If $M$ has abelian linear holonomy, then $\DD$ is a translationally isotropic 
domain.
\end{mthm}

The proof is given in Section \ref{sec_proof}.
Using Theorem \ref{thm_index_bound} we conclude:

\begin{mcor}
Let $M=\DD/\Gamma$ be a geodesically incomplete
flat pseudo-Riemannian homogeneous manifold of metric signature
$(n-s,s)$ with $s\in\{0,1,2,3\}$.
Then $\DD$ is a translationally isotropic domain.
\end{mcor}

A major consequence of Theorem \ref{thm_abhol_transiso} for
Duncan and Ihrig's classification is now:

\begin{mcor}
Taking into account the correction in Remark \ref{rem_correction},
the classification of flat homogeneous manifolds of metric
signature $(n-2,2)$
with translationally isotropic domain given by Duncan and Ihrig
\cite[Theorem 3.1]{DI3} is in fact the full classification of flat
homogeneous manifolds of metric signature $(n-2,2)$.
\end{mcor}

\begin{remark}
For signatures with $s\geq 4$, non-abelian linear holonomy groups
can occur, as Example 6.2 with signature $(4,4)$
in Baues and Globke \cite{BG} shows.
In this particular example,
$\DD$ is translationally isotropic.
However, it is not clear if the assumption of abelian
linear holonomy in
Theorem \ref{thm_abhol_transiso} can be dropped.
\end{remark}

\section{Affine and linear holonomy groups}\label{sec_holonomy}

Let $M=\DD/\Gamma$ be a flat homogeneous pseudo-Riemannian space
of signature $(n-s,s)$.
The affine holonomy group $\Gamma$ consists of affine transformations of
$\RR^n_s$ leaving $\DD$ invariant, and the group
$\Hol(\Gamma)\subset\OO_{n-s,s}$ consisting of its linear
parts is called the \emph{linear holonomy group}.\footnote{The
linear holonomy group is what differential geometers simply
refer to as the \emph{holonomy group}: The group 
generated by parallel transports of tangent vectors around loops
based at a certain point in $M$, see also Wolf \cite[Lemma 3.4.4]{wolf}.}

For $M$ to be homogeneous, it is necessary that the
centraliser of $\Gamma$ in $\Iso(\DD)$ acts transitively
on $\DD$, or in other words, the centraliser of
$\Zen_{\Iso(\RR^n_s)}(\Gamma)$ has an open orbit
$\DD\subset\RR^n_s$ (see Wolf \cite[Theorem 2.4.17]{wolf}).
We recall some properties of groups $\Gamma$ with this property
(see Wolf \cite[Section 3.7]{wolf} for the original proofs,
adapted to incomplete spaces by Duncan and Ihrig
\cite[Section 4]{DI2}).

An affine transformation $g\in\Aff(\RR^n)$ is written as
$g=(A,v)$, where $A$ is the linear part of $g$ and $v$ is its
translation part. The identity matrix is denoted by $I$.
Let $\im A$ and $\ker A$ denote the image and the kernel of
a matrix $A$, respectively.
An element $v\in\RR^n_s$ is called \emph{isotropic} if
$\langle v,v\rangle=0$, and
a vector subspace $U\subset\RR^n_s$ is called
\emph{totally isotropic} if $\langle u,v\rangle=0$ for all
$u,v\in U$.

\begin{lem}[Wolf \cite{wolf1}]\label{lem_wolf}
Let $\gamma\in\Gamma$. Then $\gamma=(I+A,v)$, where
$A^2=0$, $Av=0$, $\im A$ is totally isotropic and
$v\perp\im A$.
Moreover, $\im A=(\ker A)^\perp$ and $\ker A=(\im A)^\perp$.
\end{lem}

We define a vector subspace of $\RR^n_s$,
\[
U_\Gamma=\sum_{(I+A,v)\in\Gamma} \im A,\quad
\text{with }
U_\Gamma^\perp=\bigcap_{(I+A,v)\in\Gamma} \ker A.
\]
Then the subspace
\begin{equation}
U_0 = U_\Gamma \cap U_\Gamma^\perp
\label{eq_U0}
\end{equation}
is totally isotropic.

\begin{prop}[Wolf \cite{wolf}]\label{prop_wolf_abelian}
The following are equivalent:
\begin{enumerate}
\item
$\Hol(\Gamma)$ is abelian.
\item
If $(I+A_1,v_1),(I+A_2,v_2)\in\Gamma$, then $A_1 A_2=0$.
\item
The space $U_\Gamma$
is totally isotropic.
\item
$U_0=U_\Gamma$.
\end{enumerate}
\end{prop}

The vector space $\RR^n_s$ decomposes as
\[
\RR^n_s = U_0\oplus W_0\oplus U_0^*,
\]
where $U_0^*$ is a dual
space to $U_0$ and $W_0$ is orthogonal to $U_0$ and $U_0^*$.
It was shown in Baues and Globke \cite[Theorem 4.4]{BG} that in
a basis subordinate to this decomposition, the linear part of
$\gamma=(I+A,v)\in\Gamma$ takes the form
\begin{equation}
A=\begin{pmatrix}
0 & -B^\top \tilde{I} & C \\
0 & 0 & B \\
0 & 0 & 0
\end{pmatrix},
\label{eq_matrix}
\end{equation}
where $C$ is skew-symmetric, and the columns of $B$ are isotropic
and mutually ortho\-gonal with respect to $\tilde{I}$, the matrix
representing the restriction of $\langle\cdot,\cdot\rangle$
to $W_0$.

\begin{thm}\label{thm_index_bound}
Let $\Gamma\subset\Iso(\RR^n_s)$
be a group acting on $\RR^n_s$ whose centraliser
in $\Iso(\RR^n_s)$ has an open orbit.
If $\Hol(\Gamma)$ is not abelian, then $s\geq 4$.
\end{thm}
A weaker version of this theorem was already proved
in \cite[Theorem 5.1]{BG}, where the statement ``$s\geq 4$''
is replaced by ``$n\geq 8$''.
\begin{proof}
If $\Hol(\Gamma)$ is not abelian, then there exist elements
$(I+A_1,v_1),(I+A_2,v_2)\in\Gamma$ such that
$A_1A_2\neq0$ (Proposition \ref{prop_wolf_abelian}).
This means the respective blocks $B_1,B_2$ in (\ref{eq_matrix})
are not $0$.

According to rule (1) in Baues and Globke \cite[Section 5]{BG},
the columns of
$B_1$ and $B_2$ contain at least four linearly independent
isotropic vectors $b_1^1, b_1^2, b_2^1, b_2^2\in W_0$
satisfying $\langle b_1^i, b_1^j\rangle_{W_0}=0$ and
$\langle b_1^i,b_2^j\rangle_{W_0}\neq0$ for $i\neq j$.
In particular, $b_1^1$ and $b_1^2$ span a 2-dimensional
totally isotropic subspace $W'\subset W_0$.

Moreover, this implies that $\rk B_1=\rk B_1^\top\geq 2$, so that
$\dim U_0\geq 2$. As $W'\perp U_0$, the space
$W'\oplus U_0$ is a totally isotropic subspace of
dimension $\geq 4$. But then $s\geq 4$ holds, because $s$
is an upper bound for the dimension of totally isotropic
subspaces.
\end{proof}

%
%
%

\section{Proof of the main result}\label{sec_proof}

Let $\DD\subset\RR^n_s$ be an open domain. Its isometry group
is
\[
\Iso(\DD) = \{ g\in\Iso(\RR^n_s)\mid g.\DD\subseteq\DD\},
\]
and by
$T=\Iso(\DD)\cap\RR^n_s$
we denote the set of translations
leaving $\DD$ invariant, $T+\DD=\DD$.
If $T^\perp\subset T$, then we call $\DD$
\emph{translationally isotropic}.

\begin{lem} \label{lem_U0_in_D}
Let $U$ be a totally isotropic subspace in $\RR^n_s$.
If $U^\perp\subseteq T$, then $\DD$ is translationally isotropic.
\end{lem} 
\begin{proof}
Assume $U^\bot\subseteq T$. If some vector $v$
satisfies $v+\DD\not\subset \DD$, then $v\not\in U^\bot$.
But then $v\not\perp U\subset T$.
In particular, $v\not\in T^\perp$.
So any vector in $u\in T^\perp$ satisfies $u+\DD\subseteq\DD$,
which means that $\DD$ is translationally isotropic.
\end{proof}

In the following, we will assume $\DD$ is the development
image for a flat pseudo-Riemannian homogeneous space.
So let $\Gamma\subseteq\Iso(\RR^n_s)$ be a
discrete group acting freely and properly discontinuously on
$\RR^n_s$, and
$\DD\subset\RR^n_s$ is an open orbit of its centraliser $C =\Zen_{\Iso(\RR^n_s)}(\Gamma)$
(all this holds if $\Gamma$ is the affine holonomy group of
$M=\DD/\Gamma$).
Let $T$ be the set of translations in $\RR^n_s$ satisfying
$T+\DD=\DD$, and let
$U_\Gamma$ and $U_0$ be the subspaces of $\RR^n_s$
defined in Section \ref{sec_holonomy}.

We prove that if $\Gamma$ has abelian linear holonomy,
then $\DD$ is
translationally isotropic.

%

\begin{lem} \label{lem_U0_translations}
Identify $U_0$ with the group of translations by vectors in $U_0$.
Then $U_0\subset C \cap T$.
\end{lem} 
\begin{proof}
A translation by $u\in U_0$ is represented by $(I,u)$. If $(I+A,v)\in\Gamma$,
then
\[
(I+A,v)(I,u)=(I+A,u+Au+v)=(I+A,u+v)=(I,u)(I+A,v),
\]
where we used the fact that $U_0\subset\ker A$ by definition.
So $(I,u)\in C $ and thus $(I,u)$ is a translation leaving $ C $-orbits
invariant, implying $(I,u)\in C \cap T$.
\end{proof}

\begin{lem}\label{lem_abhol_transiso}
$\Hol(\Gamma)$ is abelian if and only if
$U_0^\bot\subseteq C $.
If this holds,
then $\DD$ is translationally isotropic.
\end{lem} 
\begin{proof}
Let $u\in U_0^\bot$. Then
\[
(I+A,v)(I,u)=(I+A,u+Au+v)=(I+A,u+v)=(I,u)(I+A,v)
\]
for all $(I+A,v)\in\Gamma$ if and only if $Au=0$ for all $(I+A,v)\in\Gamma$.
But this is equivalent to
\[
u\in\bigcap_A \ker A=U_{\Gamma}^\bot\subset U_0^\bot.
\]
As $u$ is arbitrary, this means $U_\Gamma^\perp=U_0^\perp$,
which again is equivalent to the linear holonomy of $\Gamma$ being abelian
by part (4) of Proposition \ref{prop_wolf_abelian}.
In this case, $\DD$ is translationally isotropic by Lemma \ref{lem_U0_in_D}.
\end{proof}

This concludes the proof of Theorem \ref{thm_abhol_transiso}.

%


\end{document}